\documentclass{amsart}
\usepackage{graphicx}
\usepackage{amsmath}
\usepackage{amscd}
\usepackage{mathtools}
\usepackage{enumerate}
\usepackage{tikz-cd}
\usepackage[all,cmtip]{xy}
\usepackage{tikz}
\usepackage{amssymb}
\usepackage{ascmac} 
\usepackage{color}
\usepackage{autobreak}
\usepackage{breqn}
\newtheorem{dfn}{Definition}[section]
\newtheorem{thm}[dfn]{Theorem}
\newtheorem{pro}[dfn]{Proposition}
\newtheorem{lem}[dfn]{Lemma}

\theoremstyle{definition}

\newtheorem{que}[dfn]{Question}

\newcommand{\co}{\colon}

\newcommand{\mb}{\mathbb}
\newcommand{\mr}{\mathrm}
\newcommand{\ol}{\overline}
\newcommand{\mac}{\mathcal}
\newcommand{\tb}{\textbf}
\newcommand{\ep}{\emph}

\newcommand{\X}{X_1,\ldots,X_{n+1}}

\newcommand{\rmc}{{\rm char}}
\newcommand{\da}{\dashrightarrow}

\title[Subgroups Realized by wild Galois Points]{Subgroups of the Projective Linear Group Realized by wild Galois Points}
\author{Taro Hayashi}
\author{Kashu Ito}
\author{Atsuya Nakajima}
\author{Keika Shimahara}
\address{(Taro Hayashi)
Department of Mathematical Sciences,
Ritsumeikan University,
1$-$1$-$1 Nojihigashi, Kusatsu, Shiga, 525$-$8577, Japan}
\email{haya4taro@gmail.com}
\address{(Kashu Ito)
Graduate School of Mathematical Sciences,
Ritsumeikan University,
1$-$1$-$1 Nojihigashi, Kusatsu, Shiga, 525$-$8577, Japan}
\email{ra0118vf@ed.ritsumei.ac.jp}
\address{(Atsuya Nakajima)
Graduate School of Mathematical Sciences,
Ritsumeikan University,
1$-$1$-$1 Nojihigashi, Kusatsu, Shiga, 525$-$8577, Japan}
\email{rp0111hf@ed.ritsumei.ac.jp}
\address{(Keika Shimahara)
Graduate School of Mathematical Sciences,
Ritsumeikan University,
1$-$1$-$1 Nojihigashi, Kusatsu, Shiga, 525$-$8577, Japan}
\email{ra0134hp@ed.ritsumei.ac.jp}
\date{\today}
\subjclass{Primary 14J70; Secondary 14G17, 20G15}
\keywords{Projective hypersurface; Galois cover; Wildly ramified}
\begin{document}
\maketitle
\begin{abstract}
We work over an algebraically closed field of positive characteristic. This paper investigates linear representations of Galois groups arising from wild Galois points on projective hypersurfaces. We prove that these Galois groups lift to the general linear group and act naturally on vector spaces. Furthermore, we establish necessary and sufficient conditions for subgroups of the projective linear group to be realized as Galois groups of wild Galois points. In addition, we show that projections from wild Galois points on normal hypersurfaces are necessarily wildly ramified. We provide a geometric criterion for detecting wild ramification via the fixed loci of birational automorphisms, linking group-theoretic properties to the geometry of the hypersurface.
\end{abstract}
\section{Introduction}
Let $k$ be an algebraically closed field, and
let \( \rmc(k) \) denote  the characteristic of $k$.
Let \( \mathrm{PGL}(m,k) \) denote the projective linear group over \( k \), i.e., the quotient of \( \mathrm{GL}(m,k) \) by the scalar subgroup \( \{ aI_m \mid a \in k^\times \} \) where \( I_m \) is the identity matrix. For \( A \in \mathrm{GL}(m,k) \), we denote its equivalence class in \( \mathrm{PGL}(m,k) \) by \( [A] \). We write \( e := [I_m] \in \mathrm{PGL}(m,k) \). Let 
\[
pr \colon \mathrm{GL}(m,k)\ni A\mapsto [A]\in \mathrm{PGL}(m,k)
\] 
be the natural projection, which is a group homomorphism.
We say that a subgroup \( G \subset \mathrm{PGL}(m,k) \) is \emph{\tb{liftable}} if there exists a subgroup \( G' \subset \mathrm{GL}(m,k) \) such that the restriction \( pr\mid_{G'} \colon G' \to G \) is an isomorphism ([\ref{bio:z22},\ Definition 2.1]).

For an irreducible algebraic variety $Y$ over $k$, let $k(Y)$ be the function field of $Y$. 
For a positive integer $m$, let $\mb P^m$ be the projective space of dimension $m$ over $k$.
Let $\mac X\subset\mb P^{n+1}$ be an irreducible hypersurface of degree $d\geq 3$.
Consider a point $P\in \mathbb P^{n+1}$ and the projection 
\[\pi_P\co \mac X\dashrightarrow \mb P^n\] with center $P$. 
Since $\pi_P$ is a dominant rational map,
we have an extension of function fields $\pi_P^{\ast}\co k(\mb P^n)\hookrightarrow  k(\mac X)$ with 
\[[k(\mac X)\co k(\mb P^n)]=d-m\]
where $m$ is the multiplicity of $\mac X$ at $P$. 
Note that if $m=0$, i.e. $P\not\in \mac X$, then $\pi_P$ is a finite morphism.
The point $P\in\mathbb P^{n+1}$ is called a \ep{\tb{Galois point}} for $\mac X$ if the extension $k(\mac X)/k(\mb P^n)$ is a Galois extension ([\ref{bio:my00}]).
In particular, when $p := \mathrm{char}(k) > 0$ and $d-m$ is divisible by $p$, 
we call such a Galois point a \emph{\tb{wild Galois point}}. 
Elements of ${\rm PGL}(n+2, k)$ that preserve $\mac X$ induce automorphisms of $\mac X$, and hence they induce automorphisms of $k(\mac X)$.
If the Galois group $G$ of $k(\mac X)/k(\mb P^n_k)$ is given by a subgroup of ${\rm PGL}(n+2,k)$,
then $P$ is called an \ep{\tb{extendable Galois point}} ([\ref{bio:fm14}]). If $\mac X$ is smooth, then all Galois points for $\mac X$ are extendable ([\ref{bio:acgh},\ref{bio:mm63},\ref{bio:y01g}]).
Galois points have also been investigated for reducible curves~[\ref{bio:it24}] and within arithmetic frameworks~[\ref{bio:thosn25}], making them a research theme with broad connections.

Let \( P \in \mathbb{P}^{n+1} \) be an extendable Galois point for an irreducible hypersurface \( \mac X \subset \mathbb{P}^{n+1} \) of degree $d\geq 3$, and let $m$ be the multiplicity of $\mac X$ at $P$. 
If $d-m$ is coprime to $\rmc(k)$, then the Galois group of the projection \( \pi_P\co \mac X\dashrightarrow \mb P^n \) is a cyclic group of order $d-m$ (Lemma~\ref{2.2}).
Over a field of characteristic zero, this result is already known (see [\ref{bio:ft14},\ Key\ Lemma\ 3.1]).
If \( d-m \) is divisible by \( \rmc(k) \), 
then the structure of the Galois group of $\pi_P$ has a more complicated structure.
For the cases \( m \in \{0,1\} \), the structure is already known.
\begin{thm}\label{1.15}$([\ref{bio:f07},\ {\rm Theorem}\ 2],[\ref{bio:ft14},\ {\rm Corollary}\ 1.4])$.
Let $k$ be an algebraically closed field of characteristic $p>0$,
and	let $\mac X \subset \mb P^{n+1}$ be a normal hypersurface of degree $d \geq 3$.
Let $P \in \mb P^{n+1}$ be an inner $($resp. outer$)$ Galois point for $\mac X$, $G$ be the Galois group of the projection \( \pi_P \colon \mac X \dashrightarrow \mathbb{P}^n \), and let $m$ be the multiplicity of $\mac X$ at $P$.
We write $d - m= p^u l$ where $\gcd(p,l)=1$.
If  \( m \in \{0,1\} \), then $l$ divides $p^u - 1$, and
\[
G\cong (\mathbb{Z}/p\mathbb{Z})^{\oplus u} \rtimes \mathbb{Z}/l\mathbb{Z}.
\]
\end{thm}

Over a field of characteristic zero, the Galois groups arising from extendable Galois points are closely related to the automorphism groups of hypersurfaces~([\ref{bio:fmt19},\ref{bio:hmo18},
\ref{bio:hmo22},\ref{bio:th21l},\ref{bio:th21d},\ref{bio:kty01},\ref{bio:mo15}]). Moreover, the research results of extendable Galois points are applied to the study of Galois coverings of the projective spaces~([\ref{bio:th23g},\ref{bio:th24a}]). 
In positive characteristic, the study of Galois points is initiated in~[\ref{bio:h06}]. 
In contrast, the study of Galois groups associated with extendable Galois points in positive characteristic is still at an early stage. 
There remains considerable room for further development, including their possible applications~([\ref{bio:bf22},\ref{bio:f14},\ref{bio:f22},\ref{bio:f25}]).
To extend the applicability of the theory of Galois points, it is essential to investigate in detail the Galois groups of wild Galois points, a phenomenon unique to positive characteristic. 
The aim of this paper is to address this problem.
We prove that the Galois groups arising from extendable wild Galois points lift to $\mathrm{GL}(n+2,k)$ and can be analyzed via their natural linear representations on vector spaces. 
In addition, we establish necessary and sufficient conditions for subgroups of the projective linear group to occur as Galois groups of wild Galois points.
\begin{thm}\label{1.1}
Let $k$ be an algebraically closed field of characteristic $p>0$.
Let \( \mac X \subset \mathbb{P}^{n+1} \) be an irreducible hypersurface of degree \( d \) over $k$ with an extendable Galois point  \( P \in \mathbb{P}^{n+1} \), let $m$ be the multiplicity of $\mac X$ at $P$, and let $G$ be the Galois group of the projection \( \pi_P \colon \mac X \dashrightarrow \mathbb{P}^n \).
	We assume that $p$ divides $d-m$. We write $d-m=p^ul$ where $\gcd(p,l)=1$.
Then we have the following:
\begin{enumerate}
\item[$(1)$]The number $l$ divides $p^u-1$.
\item[$(2)$]The group $G$ is isomorphic to $(\mathbb{Z}/p\mathbb{Z})^{\oplus u} \rtimes \mathbb{Z}/l\mathbb{Z}$.
\end{enumerate}
\end{thm}
Let \( \mathrm{UT}(m,k) \subset \mathrm{GL}(m,k) \) denote the set of all upper triangular matrices.
\begin{thm}\label{1.2}
Let $k$ be an algebraically closed field of characteristic $p>0$.
Let \( G \subset \mathrm{PGL}(n+2,k) \) be a subgroup isomorphic to \( (\mathbb{Z}/p\mathbb{Z})^{\oplus u} \rtimes \mathbb{Z}/l\mathbb{Z} \) where \( \gcd(p, l) = 1 \) if $l\geq2$. Then the following statements hold:
\begin{enumerate}
\item[$(1)$] The group \( G \) is liftable.
\item[$(2)$]There exists an irreducible hypersurface \( \mac X \subset \mathbb{P}^{n+1} \)  with an extendable wild Galois point $P\in\mb P^{n+1}$ such that \( G \) is the Galois group of the projection $\pi_P \colon \mac X \dashrightarrow \mathbb{P}^n $ if and only if there exists a group \( G' \subset \mathrm{GL}(n+2,k) \) such that \( pr\mid_{G'} \colon G' \to G \) is an isomorphism, and \( G' \) satisfies the following two conditions:
\begin{itemize}
\item[$(i)$] \( \dim \mathrm{Im}(A - I_{n+2}) = 1 \) for some \( A \in G' \);
\item[$(ii)$] \( \mathrm{Im}(B - I_{n+2}) = \mathrm{Im}(C - I_{n+2}) \) for all \( B, C \in G' \setminus \{I_{n+2}\} \).
\end{itemize}
\end{enumerate}
\end{thm}
In Theorem \ref{4.7}, we determine the defining equations of normal hypersurfaces with an extendable wild Galois point.

Section 2 contains preliminaries.
we provide preliminary definitions and background results on Galois points and their associated Galois groups. We focus on the case of positive characteristic, introducing the notion of extendable Galois points and describing how their Galois groups can be represented via upper triangular matrices. These foundational tools are essential for the structural analysis in the subsequent sections.
In Section 3, we prove the main theorems concerning the structure and liftability of Galois groups arising from wild Galois points. We show that such groups are semidirect products of elementary abelian $p$-groups and cyclic groups, and we give necessary and sufficient conditions for their realization via hypersurfaces. Explicit constructions of hypersurfaces with prescribed Galois groups are also provided.
In Section~4, we investigate the relation between wild Galois points and wild ramification. 
We prove that, if a hypersurface in projective space is normal, then the projection from a wild Galois point is wildly ramified in the sense of discrete valuation rings. 
In this sense, wildly ramified Galois covers of smooth curves in positive characteristic are studied, for example, in [\ref{bio:p02}]. 
As an outcome, we finally pose an open question concerning the existence of normal hypersurfaces with wild Galois points.
\section{Preliminary}
Let $X$ be an irreducible algebraic variety, and let $g$ be a birational self-map of $X$. 
Then the pullback map
\[
g^* \colon k(X) \ni f \mapsto f \circ g \in k(X)
\]
is a $k$-automorphism of the function field $k(X)$. 
In particular, the group of $k$-automorphisms $\mathrm{Aut}_k(k(X))$ is identified with 
the group of birational self-maps $\mathrm{Bir}(X)$ of $X$ via the pullback ([\ref{bio:rh}, \S1. Theorem 4.4]), and we shall make this identification throughout the paper. 

For variables $Y_1,\ldots, Y_m$ and $d\geq 0$, let 
$k[Y_1,\ldots ,Y_m]_d$ be the $k$-vector space of forms of degree $d$ in the variables $Y_1,\ldots, Y_m$.
For $A=(a_{ij})\in{\rm GL}(n+2,k)$ and
$F(X_0,\ldots,X_{n+1})\in k[X_0,\ldots,X_{n+1}]_d$, 
we define the action of $A$ on $F(X_0,\ldots,X_{n+1})$ as follows:
\[ A^{\ast}F(X_0,\ldots,X_{n+1}):=F\left(\sum_{i=1}^{n+2}a_{1i}X_{i-1},\ldots,\sum_{i=1}^{n+2}a_{n+2\,i}X_{i-1}\right).\]
Let $\mac X\subset \mb P^{n+1}$ be a hypersurface of degree $d$ defined by $F(X_0,\ldots,X_{n+1})=0$.
 If there exists $t \in k$ such that
\[A^{\ast}F(X_0,\ldots,X_{n+1})=tF(X_0,\ldots,X_{n+1}),\] then $A$ induces an automorphism of $\mac X$, which we denote by $g_A$.
Let $\pi_P\colon \mac X\dashrightarrow \mathbb P^n$ be the projection with center $P\in \mathbb P^{n+1}$.
Define
\begin{equation}\label{1}
G_{\pi_P} := \{ g \in \mathrm{Bir}(\mathcal{X}) \mid \pi_P \circ g = \pi_P \}
\end{equation}
which is a subgroup of $\mathrm{Bir}(\mathcal{X})$.  
Let $m$ be the multiplicity of $\mac X$ at $P$.
Then $P$ is a Galois point if and only if $|G_{\pi_P}|=d-m$.
The point $P$ is an extendable Galois point if and only if $G_{\pi_P}\subset \mr{PGL}(n+2,k)$.
We assume that $P=[1:0:\cdots:0]$.  
Then the restriction of the projection $\pi_P$ to $\mathcal{X}\setminus\{P\}$ is given by  
\[
\pi_P\mid_{\mathcal{X}\backslash\{P\}} \colon \mathcal{X}\backslash\{P\}\ni[X_0:X_1:\cdots:X_{n+1}]
\mapsto[X_1:\cdots:X_{n+1}] \in\mathbb{P}^n.
\]
We define the subgroup $\mr{UT}(*,I_{n+1})$ of $\mr{GL}(n+2,k)$ by  
\[
\mr{UT}(*,I_{n+1}) := 
\left\{
\begin{pmatrix}
	a & \overline{b} \\
	0 & I_{n+1}
\end{pmatrix}\in \mr{GL}(n+2,k)
\;\middle|\;
a \in k^{\times},\; \overline{b} \in k^{n+1}
\right\}
\]
where $k^{\times}:=k\backslash\{0\}$.
\begin{lem}\label{2.1}$([\ref{bio:y03},\ p.528])$.
Let $P:=[1:0:\cdots:0]\in\mb P^{n+1}$ be a point, 
let $\mac X\subset \mb P^{n+1}$ be an irreducible hypersurface of degree $d\geq3$, and
let	$\alpha\in {\rm PGL}(n+2,k)$ be an element  such that $\alpha\in G_{\pi_P}$.
Then there is a unique matrix $A\in \mathrm{UT}(*,I_{n+1})$ such that $\alpha=[A]$.
\end{lem}
\begin{proof}
Let $B:=(b_{ij})\in{\rm GL}(n+2,k)$ be a matrix such that $\alpha=[B]$.
Since $\pi_P\mid_{\mac X\backslash \{P\}}([X_0:\cdots :X_{n+1}])=[X_1:\cdots:X_{n+1}]$,
and $\pi_P\circ \alpha=\pi_P$, we have
\[[\sum_{i=1}^{n+2}b_{2i}X_{i-1}:\cdots :\sum_{i=1}^{n+2}b_{n+2\,i}X_{i-1}]=[X_1:\cdots:X_{n+1}]\]
for $[X_0:\cdots :X_{n+1}]\in  \mac X\backslash\{P\}$.
It follows that
\begin{equation*}
\left(\sum_{i=1}^{n+2}b_{u+1\,i}X_{i-1}\right)X_v-\left(\sum_{i=1}^{n+2}b_{v+1\,i}X_{i-1}\right)X_u=0
\end{equation*}
for any point $[X_0:\cdots :X_{n+1}]\in \mac X$ and $u,v=1,\ldots, n+1$.
Since the degree of $\mac X$ is at least $3$, 
it follows that the above relation holds as an identity of polynomials in $X_0,\dots,X_{n+1}$.
Consequently, we obtain $b_{22}=b_{ii}$ for $i=2,\ldots, n+2$ and $b_{ij}=0$ for $2\leq i,j\leq n+2$ and $i\neq j$.
We set $A:=a_{22}^{-1}B$ and $a_{1i}:=b_{1i}a_{22}^{-1}$ for $i=1,\ldots, n+2$.
Then
\[
A=\begin{pmatrix} 
a_{11} & a_{12} & \dots  & a_{1n+2} \\
0 & 1 & \dots  & 0\\
\vdots & \vdots & \ddots & \vdots \\
0 & 0& \dots  & 1
\end{pmatrix}\in\mathrm{UT}(*,I_{n+1})\quad{\rm and}\quad \alpha=[A].
\]
Let \( C \in \mathrm{UT}(*, I_{n+1}) \) be a matrix such that \([A] = [C]\). 
Then there exists \( t \in k^{\times} \) with \( A = tC \). 
Since the \((i,i)\)-entries of both \( A \) and \( C \) are equal to \(1\) for \( i = 2, \ldots, n+2 \), 
it follows that \( t = 1 \). 
Hence, the uniqueness of \( A \) is established.
\end{proof}
Let $G$ be a group. For a subset $S \subset G$, we write $\langle S \rangle$ for the subgroup of $G$ generated by $S$.
For an integer $l\geq 2$, let $e_l$ denote a primitive $l$-th root of unity in the field $k$.
\begin{lem}\label{2.2}
Let $k$ be an algebraically closed field, and let $p:=\rmc(k)$.
Let $P:=[1:0:\cdots:0]\in\mb P^{n+1}$ be a point, 
let $\mac X\subset \mathbb P^{n+1}$ be an irreducible hypersurface of degree $d\geq3$, and let $G_{\pi_P}$ be the group defined in $(\ref{1})$.
We assume that $G_{\pi_P}\subset \mathrm{PGL}(n+2,k)$.
Then we have the following:
\begin{enumerate}
\item[$(1)$]There exists a group \( G \subset \mathrm{UT}(*,I_{n+1}) \) such that the homomorphism $pr\mid_{G} \colon G \to G_{\pi_P}$ is an isomorphism.
\item[$(2)$]We assume that $p>0$ and the order $l:=|G_{\pi_P}|$ of $G_{\pi_P}$ is coprime to $p$. 
Then $G_{\pi_P}$ is a cyclic group. Moreover, there exists a matrix 
$B \in \mathrm{UT}(*, I_{n+1})$ such that
\[
\beta G_{\pi_P} \beta^{-1}
= \left\langle 
\left[\begin{pmatrix}
e_l & 0 \\
0 & I_{n+1}
\end{pmatrix}\right]
\right\rangle
\]
where $\beta:=[B]\in\mathrm{PGL}(n+2,k)$.
\end{enumerate}
\end{lem}
\begin{proof}
We show Lemma \ref{2.2} (1).
By Lemma \ref{2.1}, for $g\in G_{\pi_P}$, there is a unique matrix  $A(g)\in \mr{UT}(*,I_{n+1})$ such that 
$g=[A(g)]$.
For each $g \in G_{\pi_P}$, let $A(g)$ denote the associated matrix, and denote by $A(g)_i \in k$ its $(1,i)$-entry for $i=1,\ldots,n+2$.
Let $G\subset \mr{UT}(*,I_{n+1})$ be the group generated by $\{\,A(h)\,\}_{h\in G_{\pi_P}}$.
Since 
\[A(g)A(h)=\begin{pmatrix} 
a(g)_1a(h)_1& a(g)_1a(h)_2 +a(g)_2& \dots  & a(g)_1a(h)_{n+2}+a(g)_{n+2}\\
0 & 1 & \dots  & 0\\
\vdots & \vdots & \ddots & \vdots \\
0 & 0& \dots  & 1
\end{pmatrix}\]
is contained in $\mathrm{UT}(*,I_{n+1})$,
we have $A(gh)=A(g)A(h)$ for $g,h\in G_{\pi_P}$.
Since 
\[A(g)^{-1}=
\begin{pmatrix} 
\frac{1}{a(g)_1} & -\frac{a(g)_{2}}{a(g)_1} & \dots  & -\frac{a(g)_{n+2}}{a(g)_1} \\
0 & 1 & \dots  & 0\\
\vdots & \vdots & \ddots & \vdots \\
0 & 0& \dots  & 1
\end{pmatrix}\in\mathrm{UT}(*,I_{n+1}),\]
we get $A(g^{-1})=A(g)^{-1}$ for $g\in G_{\pi_P}$.
Thus, $pr\mid_{G}\co G\to G_{\pi_P}$ is an isomorphism.

We show Lemma \ref{2.2} (2).
We assume that $p>0$ and $\gcd(|G_{\pi_P}|,p)=1$. 
Let $G\subset \mathrm{UT}(*,I_{n+1})$ be a group such that $pr\mid_{G}\co G\to G_{\pi_P}$ is an isomorphism.
Let $A:=\begin{pmatrix} 
a_{11} & a_{12} & \dots  & a_{1n+2} \\
0 & 1 & \dots  & 0\\
\vdots & \vdots & \ddots & \vdots \\
0 & 0& \dots  & 1
\end{pmatrix}\in G$.
We assume that $a_{11}=1$.
Then
\begin{equation*}
\begin{split}
A^p&=\begin{pmatrix} 
	1 & pa_{12} & \dots  & pa_{1n+2} \\
	0 & 1 & \dots  & 0\\
	\vdots & \vdots & \ddots & \vdots \\
	0 & 0& \dots  & 1
\end{pmatrix}\\
&=I_{n+2}.
\end{split}
\end{equation*}
Since $\gcd(|G_{\pi_P}|,p)=1$, we get that $A=I_{n+2}$.  
Thus, a homomorphism
\[\varphi\co G\ni A\mapsto a_{11}\in k^{\times}\]
is injective.
Since $\mr{im}(\varphi)$ is a finite subgroup of $k^{\times}$, 
$\mr{im}(\varphi)$ is a cyclic group. Thus, $G_{\pi_P}$ is a cyclic group.

Let $A=\begin{pmatrix}
	a_{11} & a_{12} & \dots  & a_{1n+2} \\
	0 & 1 & \dots  & 0\\
	\vdots & \vdots & \ddots & \vdots \\
	0 & 0& \dots  & 1
\end{pmatrix}\in G$ be a generator of $G$.
Since $G$ is a cyclic group of order $l$, $a_{11}$ is a primitive \(l\)-th root of unity.
We set
\[
B:=\begin{pmatrix}
	1 & \frac{a_{12}}{a_{11}-1} & \dots  & \frac{a_{1n+2}}{a_{11}-1} \\
	0 & 1 & \dots  & 0\\
	\vdots & \vdots & \ddots & \vdots \\
	0 & 0& \dots  & 1
\end{pmatrix}\in\mathrm{UT}(*,I_{n+1}).
\]
Then we have
\[
BAB^{-1}=
\begin{pmatrix}
	a_{11} & 0 \\
	0 & I_{n+1}
\end{pmatrix}.
\]
We set $\beta:=[B]\in\mathrm{PGL}(n+2,k)$.
Then $\beta G_{\pi_P} \beta^{-1}
= \left\langle 
\left[\begin{pmatrix}
	e_l & 0 \\
	0 & I_{n+1}
\end{pmatrix}\right]
\right\rangle$.
\end{proof}
We establish necessary and sufficient conditions for the existence of an extendable Galois point on a hypersurface. 
This gives a partial extension of the corresponding statement in $[\ref{bio:ft14},\ \mathrm{Key\ Lemma}\ 3.1]$ to the case of positive characteristic, 
and the proof is carried out in a similar manner.
\begin{thm}\label{2.3}
Let $k$ be an algebraically closed field of characteristic $p\geq0$, and
let $\mac X\subset \mb P^{n+1}$ be an irreducible hypersurface of degree $d\geq 3$.
Let $P\in\mathbb P^{n+1}$ be a point, let $G_{\pi_P}$ be the group defined in $(\ref{1})$, and let $m$ be the multiplicity of $\mac X$ at $P$.
We assume that  $p=0$ or $\gcd(p,d-m)=1$.
Then $P$ is an extendable Galois point of $X$ if and only if by replacing the coordinate system if necessary, $P=[1:0:\cdots:0]$ and  $\mac X$ is defined by
\[ F_m(\X)X_0^{d-m}+F_d(\X)=0\]
where  $F_u(\X)\in k[\X]_u$ for $u=m,d$.
Furthermore, when the defining equation takes this form, 
\[G_{\pi_P}=\left\langle 
\left[\begin{pmatrix}
e_{d-m}&0\\
0&I_{n+1}
\end{pmatrix}\right]\right\rangle.
\]
\end{thm}
\begin{proof}
Since the case of characteristic zero was established in $[\ref{bio:ft14},\ \mathrm{Key\ Lemma}\ 3.1]$, we assume that $p>0$ and $\gcd(p,d-m)=1$.
We set $A:=
\begin{pmatrix}
e_{d-m} & 0 \\
0 & I_{n+1}
\end{pmatrix}$, and 
$\alpha:=[A]\in\mathrm{PGL}(n+2,k)$.
We show the only if part.
We assume that, after a suitable change of coordinates, 
 $P=[1:0:\cdots:0]$ and
$\mathcal{X}$ is defined by
\[
F_m(\X)X_0^{\,d-m}+F_d(\X)=0,
\]
where $F_u(\X) \in k[\X]_u$ for $u=m,d$.  
From the form of this defining equation, we see that $\alpha$
is an automorphism of $\mathcal{X}$ of order $d-m$.  
Since $P=[1:0:\cdots:0]$, we have $\alpha \in G_{\pi_P}$.  
As the projection $\pi_P \colon \mathcal{X} \dashrightarrow \mathbb{P}^n$ has degree $d-m$, $G_{\pi_P}=\langle \alpha\rangle$.
It follows that $P$ is a Galois point.  
Since $\alpha \in \mathrm{PGL}(n+2,k)$, the point $P$ is an extendable Galois point for $\mathcal{X}$.

We show the if part.
By replacing the coordinate system if necessary, 
$P=[1:0:\cdots:0]$. 
Since $P$ is an extendable Galois point for $\mac X$, by Lemma \ref{2.2} $(2)$ we may assume that 
 $G_{\pi_P}=\langle \alpha \rangle$.
Since the multiplicity of \(\mathcal{X}\) at \(P\) is \(m\), 
\[
F(X_0,\ldots, X_{n+1})
= \sum_{i=0}^{d-m} F_{d-i}(X_1,\ldots, X_{n+1}) X_0^{i},
\]
where \(F_{d-i}(X_1,\ldots,X_{n+1}) \in k[X_1,\ldots,X_{n+1}]_{d-i}\) for \(i=0,\ldots,d-m\). 
In particular, \(F_{m}(X_1,\ldots,X_{n+1})\neq 0\).  
Since \(\mathcal{X}\) is irreducible, \(F_d(X_1,\ldots,X_{n+1}) \neq 0\).  
Since $\alpha$ is an automorphism of \(\mathcal{X}\), and \(F_d(X_1,\ldots,X_{n+1})\neq 0\), it follows that  
\[
A^*F(X_0,\ldots, X_{n+1})=F(X_0,\ldots, X_{n+1}).
\]
On the other hand, we compute
\[
A^*F(X_0, \ldots, X_{n+1})
= \sum_{i=0}^{d-m} F_{d-i}(X_1,\ldots, X_{n+1})(e_{d-m} X_0)^i.
\]
Since this is equal to \(F(X_0,\ldots, X_{n+1})\), by comparing coefficients, we get that 
\[
F_{d-i}(X_1,\ldots, X_{n+1})
=e_{d-m}^iF_{d-i}(X_1,\ldots, X_{n+1})
\]
for $i=1,\ldots,d-m-1$.
Thus,
we conclude that
\[
F_{d-i}(X_1,\ldots, X_{n+1})=0
\]
for  $i=1,\ldots,d-m-1$. Thus the defining equation takes the desired form.
\end{proof}
From Theorem \ref{2.3}, in order to determine whether an irreducible hypersurface 
$\mathcal{X} \subset \mathbb{P}^{n+1}$ admits an extendable Galois point, 
the problem reduces to checking whether $\mathcal{X}$ admits an automorphism 
that is conjugate in $\mathrm{PGL}(n+2,k)$ to a matrix of the form
\[
\begin{pmatrix} 
	a & 0 \\
	0 & b I_{n+1}
\end{pmatrix}
\]
where $\frac{b}{a} = e_l$. 
In characteristic zero, when a hypersurface is smooth, 
the existence of such an automorphism can be determined 
from its fixed point set and the order of the automorphism, as we explain below.
\begin{thm}$([\ref{bio:bb16n}],\ \mr{cf.\ Corollaries}\ 24\ \mr{and}\ 32\ \mr{in\ the\ preprint}\ [\mr{arXiv}:1510.06192])$.\label{2.4}
Let $k$ be an algebraically closed field of characteristic zero. 
Let $\mac X$ be a smooth curve of degree $d\geq 5$ in $\mathbb P^2$ over $k$, and let $g\in{\rm Aut}(\mac X)$ be an automorphism of order $u(d-1)$ $(resp.\ ud)$ for $u\geq 1$. 
If $u\geq2$, then $\mac X$ has an inner $(resp.\ outer)$ Galois point $P$ and $G_{\pi_P}=\langle g^u\rangle$. 
\end{thm}
\begin{thm}\label{thm:6}$([\ref{bio:th21l},\ \mr{Theorems}\ 1.7,\,1.8,\ \mr{and}\ 1.9])$.
Let $k$ be an algebraically closed field of characteristic zero. 
Let $\mac X\subset\mathbb P^{n+1}$ be a smooth hypersurface of degree $d\geq 4$ over $k$, and let $g$ be an automorphism of $\mathcal X$ for $n\geq1$.
\begin{enumerate}
\item[$(1)$]We assume that $n=1$ and ${\rm ord}(g)=d-1$ $($resp. ${\rm ord}(g)=d)$.
Then $|{\rm Fix}(g)|\not=2$ $($resp. ${\rm Fix}(g)\not=\emptyset)$
if and only if $\mac X$ has an inner $($resp. outer$)$ Galois point $P$ and $G_{\pi_P}=\langle g\rangle$.
\item[$(2)$] We assume that $n=2$ and ${\rm ord}(g)=d-1$.
Then ${\rm Fix}(g)$ contains a curve $C'$ which is not a smooth rational curve if and only if $X$ has an inner Galois point $P$ and $G_{\pi_P}=\langle g\rangle$.
\item[$(3)$] We assume that $n\geq3$ and ${\rm ord}(g)=d-1$.
Then ${\rm Fix}(g)$ has codimension $1$ in $X$ if and only if $X$ has an inner Galois point $P$ and $G_{\pi_P}=\langle g\rangle$.
\item[$(4)$] We assume that $n\geq2$ and ${\rm ord}(g)=d$.
Then ${\rm Fix}(g)$ has codimension $1$ in $X$ if and only if $X$ has an outer Galois point $P$ and $G_{\pi_P}=\langle g\rangle$.
\end{enumerate}
\end{thm}
It is important to emphasize that, when $n \geq 2$, 
the presence of an automorphism of large order does not by itself guarantee the existence of a Galois point. 
Concrete counterexamples can be found in situations similar to those described in Theorem~\ref{2.4} 
([\ref{bio:th21l},\ Examples 2.9 and 3.9]). 
To determine whether a smooth hypersurface admits a Galois point by using automorphisms, 
one must also examine the fixed point set of the automorphism 
([\ref{bio:th21l},\ Theorem 1.10]).
In [\ref{bio:th23r}], it is shown that the Galois property of a given point $P\in\mathbb P^{n+1}$ can also be determined geometrically by examining the branching divisors and ramification indices of the projection with center $P$.

As stated in Theorem \ref{2.3}, the defining equations of irreducible hypersurfaces 
admitting extendable Galois points in characteristic zero, or non-wild Galois points in positive characteristic, 
are determined. 
In contrast, the defining equations of irreducible hypersurfaces with wild Galois points 
are still largely unknown. 
So far, explicit studies exist only in two cases: inner wild Galois points for smooth plane quartics 
in characteristic $3$, and extendable outer wild Galois points when $d = p^l$ with $l \geq 1$.
\begin{thm}\label{2.5}$([\ref{bio:f06},\ {\rm Proposition}\ 1])$.
Let $k$ be an algebraically closed field of $\rmc(k)=3$, and let $[X:Y:Z]$ be the coordinate system of $\mb P^2$.
A smooth quartic curve $C\subset\mb P^2$ has an inner Galois point if and only if $C$ is projectively equivalent to the curve defined by either of the following two forms:
\begin{enumerate}
\item[$(1)$]$X^{3}Z-XZ^{3}+Y^{4}+a_{3}Y^{3}Z+a_{2}Y^{2}Z^{2}+a_{1}YZ^{3}$,
\item[$(2)$]$X^{3}Y-XYZ^{2}+a_{3}Y^{3}Z+a_{2}Y^{2}Z^{2}+a_{1}YZ^{3}+Z^{4}$,
\end{enumerate}
where $a_{i}(i=1,2,3)$ are constants, and in the case $(2)$ we require $a_{3}\neq 0$ . 
Moreover, the dual map of $C$ is inseparable onto its dual if and only if $C$ is in the case $(1)$ with $a_{2}=0$.
\end{thm}
\begin{thm}\label{2.6}$([\ref{bio:f142},\ {\rm Proposition}\ 1])$.
Let $k$ be an algebraically closed field of characteristic $p>0$, and let $[X:Y:Z]$ be the coordinate system of $\mb P^2$.
Let $C \subset \mathbb{P}^2$ be an irreducible curve of degree $p^e$. Assume that $P =[1 : 0 : 0]\in\mb P^2$ is an extendable outer Galois point, and let 
\[
g_0(x, y) = \prod_{\tilde{\sigma} \in \tilde{G}_P} \tilde{\sigma}^* x = \prod_{\sigma \in G_P} (x + a_{12}(\sigma)y + a_{13}(\sigma)),
\]
where $x:=\frac{X}{Z}$ and $y:=\frac{Y}{Z}$.
Then, we have the following:
\begin{enumerate}
		\item[$(1)$] $G_{\pi_P} \cong (\mathbb{Z}/p\mathbb{Z})^e$.
		\item[$(2)$] $g_0(x, y) \in K[y][x]$ has only terms of degree a power of $p$ in the variable $x$.
		\item[$(3)$] The defining equation of $C$ on the affine space $Z\neq0$ is of the form $g_0(x, y) + h(y) = 0$.
\end{enumerate}
\end{thm}
In Theorem \ref{4.7}, we determine the defining equations of normal hypersurfaces possessing extendable wild Galois points, thereby generalizing the above results.
\section{Proof of main theorems}
Throughout this paper, we work over an algebraically closed field $k$ of positive characteristic $p>0$. 
Unless otherwise stated, projective spaces and hypersurfaces are defined over $k$.
We begin the preparation for the proof of Theorem~\ref{1.1} and the if part of Theorem~\ref{1.2}~(2).
\begin{lem}\label{4.0}
Let $P:=[1:0:\cdots:0]\in\mb P^{n+1}$ be a point, 
let $\mac X\subset \mb P^{n+1}$ be an irreducible hypersurface of degree $d\geq3$, and let $G_{\pi_P}$ be the group defined in $(\ref{1})$.
We assume that there exists an element $g\in G_{\pi_P}$ such that $g\in\mathrm{PGL}(n+2,k)$.
Then we have the following:
\begin{enumerate}
\item[$(1)$]If \(g^{p^l} = e\) for some $l\geq1$, then $g^p=e$. 
\item[$(2)$]Let $A\in \mathrm{UT}(*,I_{n+1})$ be a matrix such that $g=[A]$.
Then $g^p=e$ if and only if the \((1,1)\)-entry of \(A\) equals \(1\).
\end{enumerate}
\end{lem}
\begin{proof}
By Lemma \ref{2.1}, for $g\in G_{\pi_P}\cap \mathrm{PGL}(n+2,k)$, there is a unique matrix  $A\in \mr{UT}(*,I_{n+1})$ such that 
$g=[A]$.
We denote by $a_{1i} \in k$ its $(1,i)$-entry for $i=1,\ldots,n+2$.
We assume that $g^{p^l}=e$ for some $l\geq1$.
	Then $A^{p^l} = t I_{n+2} \quad \text{for some } t \in k^{\times}$.
Since the $(i,i)$-entry of $A^{p^l}$ equals $1$ for $i=2,\ldots,n+2$, it follows that $t=1$. 
Since the $(1,1)$-entry of $A^{p^l}$ is given by $a_{11}^{p^l}$, we obtain $a_{11}^{p^l}=1$.
	Since $p=\rmc(k)$, we obtain $a_{11}=1$.
Since $A^p=
\begin{pmatrix} 
1 & pa_{12} & \dots  & pa_{1n+2} \\
0 & 1 & \dots  & 0\\
\vdots & \vdots & \ddots & \vdots \\
0 & 0& \dots  & 1
\end{pmatrix}$ and $p=\rmc(k)$, we get  $A^p=I_{n+2}$. Thus, $g^p=e$.
Conversely, if $a_{11}=1$, then by the above $g^p=e$.
\end{proof}
Theorem~\ref{1.1} is derived from Theorem~\ref{4.1} (1). The if part of 
Theorem~\ref{1.2} (2) is obtained from Theorem~\ref{4.1} (2).
\begin{thm}\label{4.1}
Let \( \mathcal{X} \subset \mathbb{P}^{n+1} \) be an irreducible hypersurface of degree \( d \) with an extendable Galois point \( P \in \mathbb{P}^{n+1} \),
 and let $G_{\pi_P}$ be the Galois group of the projection $\pi_P\colon \mathcal X\dashrightarrow \mathbb P^n$.
Let $m$ be the multiplicity \( m \) of $\mac{X}$ at \( P \).
We assume that \( d - m \) is a multiple of \( p \).
We set $d-m=p^ul$ with $\gcd(p,l)=1$.
Then we have the following:
\begin{enumerate}
\item[$(1)$]The number $l$ divides $p^u-1$, and $G\cong (\mathbb{Z}/p\mathbb{Z})^{\oplus u} \rtimes \mathbb{Z}/l\mathbb{Z}$.
\item[$(2)$]There exist a subgroup $G \subset \mathrm{UT}(*,I_{n+1})$ and an element $\alpha \in \mathrm{PGL}(n+2,k)$ such that the homomorphism
$pr\mid_G \colon G \longrightarrow \alpha G_{\pi_P} \alpha^{-1}$
is an isomorphism, and
\[
G = \Biggl\langle 
\begin{pmatrix}
e_l & 0 \\
0 & I_{n+1}
\end{pmatrix},\ G_p 
\Biggr\rangle
\]
where $G_p$ is a Sylow $p$-subgroup of $G$. 
Moreover, $G$ satisfies the following conditions:
\begin{itemize}
\item[$(i)$]\( \dim \mathrm{im}(A - I_{n+2}) = 1 \) for some \( A \in G \);
\item[$(ii)$]\( \mathrm{im}(B - I_{n+2}) = \mathrm{im}(C - I_{n+2}) \) for all \( B, C \in G \setminus \{I_{n+2}\} \).
\end{itemize}
\end{enumerate}
\end{thm}
\begin{proof}
Since $P$ is an extendable Galois point, $G_{\pi_P}\subset \mathrm{PGL}(n+2,k)$.
By replacing the coordinate system if necessary, we may assume that $P = [1 : 0 : \cdots : 0]$.
By Lemma \ref{2.2} (1), there exists a group $G\subset \mr{UT}(*,I_{n+1})$ such that $pr\mid_{G}\co G\to G_{\pi_P}$ is an isomorphism.
For $h\in G$, we write
\[
h=\begin{pmatrix}
h_{11} & h_{12} & \dots  & h_{1n+2} \\
0 & 1 & \dots  & 0\\
\vdots & \vdots & \ddots & \vdots \\
0 & 0& \dots  & 1
\end{pmatrix}
\]
where \(h_{11},\ldots,h_{1n+2}\in k\).
Then
\[
\varphi\co G\ni h \mapsto h_{11}\in k^{\times}
\]
is a group homomorphism such that for $g\in G$, $g\in \mr{ker}(\varphi)$ if and only if $g_{11}=1$.
Let $G_p$ be a Sylow $p$-subgroup of $G$.
We take an element $h\in G_p$.
Then $h^{p^e}=I_{n+2}$ for some $e\geq 1$.
By Lemma \ref{4.0} (1), $h^p=I_{n+2}$.
By Lemma \ref{4.1} (2), $h_{11}=1$.
Thus, $\mr{ker}(\varphi)=G_p$.
For $g,h\in \mr{ker}(\varphi)$,
\[
gh=\begin{pmatrix} 
1 & g_{12}+h_{12} & \dots  & g_{1n+2}+h_{1n+2} \\
0 & 1 & \dots  & 0\\
\vdots & \vdots & \ddots & \vdots \\
0 & 0& \dots  & 1
\end{pmatrix}=hg.
\]
Thus, $\mr{ker}(\varphi)\cong (\mathbb Z/p\mathbb Z)^{\oplus u}$.
Since $k$ is a field of characteristic $p>0$, 
$\mr{im}(\varphi)$ is a cyclic group of order $l'$ where $\gcd(p,l')=1$.
Since $|G'|=p^ul$ and $|\mr{ker}(\varphi)|=p^u$,
we get that $l=l'$.
Since $\mr{im}(\varphi)$ is a cyclic group of order $l$ and $\gcd(p,l)=1$,
there exists an element $g\in G'$  such that 
\[
g=\begin{pmatrix}
e_l & g_{12} & \dots  & g_{1n+2} \\
0 & 1 & \dots  & 0\\
\vdots & \vdots & \ddots & \vdots \\
0 & 0& \dots  & 1
\end{pmatrix}
\]
where $e_l$ is a primitive $l$-th root of unity.
Since 
\begin{equation*}
\begin{split}
g^l&=\begin{pmatrix}
	1 & g_{12}(\sum_{i=0}^{l-1}e_l^i) & \dots  & g_{1n+2}(\sum_{i=0}^{l-1}e_l^i)\\
	0 & 1 & \dots  & 0\\
	\vdots & \vdots & \ddots & \vdots \\
	0 & 0& \dots  & 1
\end{pmatrix}\\
&=I_{n+2},
\end{split}
\end{equation*}
we get that 
\[
\{I_{n+2}\}\longrightarrow \ker(\varphi)\longrightarrow G\longrightarrow \mathrm{im}(\varphi)\longrightarrow \{1\}
\]
is split, and hence
\[
G\cong \ker(\varphi)\rtimes \mathrm{im}(\varphi).
\]
Consequently, we obtain
\[
G_{\pi_P}\cong (\mathbb{Z}/p\mathbb{Z})^{\oplus u}\rtimes \mathbb{Z}/l\mathbb{Z}.
\]
Since for $h\in \ker(\varphi)$,
\[
gh=\begin{pmatrix} 
e_l & e_lh_{12} +g_{12}& \dots  & e_lh_{1n+2}+g_{1n+2}\\
0 & 1 & \dots  & 0\\
\vdots & \vdots & \ddots & \vdots \\
0 & 0& \dots  & 1
\end{pmatrix}
\]
and
\[
hg=\begin{pmatrix} 
e_l & h_{12} +g_{12}& \dots  &h_{1n+2}+g_{1n+2}\\
0 & 1 & \dots  & 0\\
\vdots & \vdots & \ddots & \vdots \\
0 & 0& \dots  & 1
\end{pmatrix},
\] we have $gh\neq hg$, and hence $ghg^{-1}\neq h$.
It follows that the number of conjugacy classes in $\ker(\varphi)$ 
under the action of the cyclic group generated by $g$ is 
\[
1 + \frac{p^u - 1}{l}.
\]
Thus, $p^u - 1$ is divisible by $l$.
Moreover,
we set 
\[
B:=
\begin{pmatrix}
1 & \frac{g_{12}}{g_{11}-1} & \cdots & \frac{g_{1n+2}}{g_{11}-1} \\
0 & 1 & \cdots & 0\\
\vdots & \vdots & \ddots &\vdots\\
0 & 0 &  \cdots & 1
\end{pmatrix}.
\]
Then
\[BgB^{-1}=
\begin{pmatrix}
e_l & 0\\
0 & I_{n+1}
\end{pmatrix}\]
and 
$BGB^{-1}\subset \mr{UT}(*,I_{n+1})$.
We set $\beta:=[B]\in\mathrm{PGL}(n+2,k)$.
Then $pr\mid_{BGB^{-1}}\colon BGB^{-1}\to \beta G_{\pi_P}\beta^{-1}$ is an isomorphism, and $BGB^{-1}$ is generated by a Sylow $p$-subgroup of $BGB^{-1}$ and  
$\begin{pmatrix}
e_l & 0\\
0 & I_{n+1}
\end{pmatrix}$.
Since $BGB^{-1}\subset \mr{UT}(*,I_{n+1})$, it is straightforward to verify that \( BGB^{-1} \) satisfies conditions~$(i)$ and~$(ii)$ of Theorem \ref{4.1} $(1)$.
\end{proof}
From this point, we begin the preparation for the proof of the only if part of Theorem~\ref{1.2}~(2).
\begin{lem}\label{4.2}
Let $A\in \mathrm{GL}(m,k)$ be a matrix such that $A^p=I_m$. 
Then all eigenvalues of $A$ are $1$.
\end{lem}
\begin{proof}
Let $\lambda\in k$ be an eigenvalue of $A$.
Since $A^p=I_m$, we have $\lambda^p=1$.
Since \( \rmc(k)=p \), \( \lambda =1 \).
\end{proof}
\begin{lem}\label{4.3}
Let $A,B,C\in \mathrm{GL}(m,k)$ be matrices such that $A^p=B^p=I_m$ and $[CA]=[BC]$ in $\mathrm{PGL}(m,k)$.
Then $CA=BC$.
\end{lem}
\begin{proof}
Since $[CA]=[BC]$, there exists an element $\lambda\in k^{\times}$ such that 
\[
CAC^{-1}=\lambda B.
\]
By Lemma \ref{4.2}, all eigenvalues of $A$ and $B$ are 1.
Then all eigenvalues of $\lambda B$ are $\lambda$.
Since the matrices $A$ and $CAC^{-1}$ have the same characteristic polynomial, 
all eigenvalues of $CAC^{-1}$ are 1.
Since $CAC^{-1}=\lambda B$, we get that $\lambda=1$.
Therefore, $CA=BC$.
\end{proof}
\begin{lem}\label{4.4}
Let $\alpha\in {\rm PGL}(m,k)$ be an element such that $\alpha^p=e$. 
Then there exists a matrix $A\in\mathrm{GL}(m,k)$ such that $\alpha=[A]$ and $A^p=I_m$.
\end{lem}
\begin{proof}
Let $C\in {\rm GL}(m,k)$ be a matrix such that $\alpha=[C]$. Since $\alpha^p=e$, there exists a non-zero element $\lambda\in k$ such that $C^p=\lambda I_m$.
Let $\tau\in k$ be an element such that $\tau^p=\lambda$.
We set $A:=\tau^{-1}C\in\mathrm{GL}(m,k)$.
Then $\alpha=[A]$ and $A^p=I_m$
\end{proof}
An \emph{\tb{unit upper triangular matrix}} is an upper triangular matrix whose diagonal entries are all $1$. Let \( \mathrm{UUT}(m,k) \subset \mathrm{GL}(m,k) \) denote the set of all unit upper triangular matrices.
Theorem~\ref{1.2} (1) is obtained from Proposition~\ref{4.5}.
\begin{pro}\label{4.5}
Let $G\subset \mr{PGL}(m,k)$ be a finite group.
\begin{enumerate}
\item[$(1)$]If $G\cong (\mathbb{Z}/p\mathbb{Z})^{\oplus u}$, then $G$ is liftable. In particular, there exists a subgroup $G_0\subset \mathrm{UUT}(m,k)$ and an element $\alpha\in \mathrm{PGL}(m,k)$ such that $pr\mid_{G_0}\co G_0\rightarrow \alpha G\alpha^{-1}$ is an isomorphism. 
\item[$(2)$]If $G\cong (\mathbb{Z}/p\mathbb{Z})^{\oplus u}\rtimes\mathbb Z/l\mathbb Z$ where $\gcd(p,l)=1$, then $G$ is liftable. 
\end{enumerate}
\end{pro}
\begin{proof}
We show Proposition \ref{4.5} $(1)$.
We assume that $G\cong (\mathbb{Z}/p\mathbb{Z})^{\oplus u}$.
By Lemma \ref{4.4}, for any $g\in G$, we can choose a representative matrix $A_{g}\in \textrm{GL}(m,k)$ such that $g=[A_{g}]$ and $A_g^p=I_m$. 
We set
\[
\widetilde{G}:=\{A_{g}\in\textrm{GL}(m,k)\mid g\in G\ \textrm{where}\ g=[A_g]\ \textrm{and}\ A_g^p=I_m\},
\]
and
\[
G_0:=\langle \widetilde{G}\rangle\subset \mathrm{GL}(m,k).
\]
By Lemma \ref{4.3}, $A_{g_{1}}A_{g_{2}}=A_{g_{2}}A_{g_{1}}$ for $g_1,g_2\in G$.
Therefore, $G_0$  is an abelian group.
Let
$\phi:=pr\mid_{G_0}\co G_{0}\ni A \mapsto[A]\in G$ be the homomorphism.
Then $\phi$ is surjective, and $\mr{ker}(\phi)=G_{0}\cap \{\lambda I_m\mid \lambda\in k^{\times}\}$.
Since $G\cong (\mathbb{Z}/p\mathbb{Z})^{\oplus u}$, and $A_g^p=I_{m}$ for $g\in G$, we obtain
\[
G_0\cong (\mathbb{Z}/p\mathbb{Z})^{\oplus t}\quad {\rm and}\quad \mr{ker}(\phi)\cong (\mathbb{Z}/p\mathbb{Z})^{\oplus s}
\] where $u\leq t$ and $s\leq t$.
Since $\rmc(k)=p$, we get that $\mr{ker}(\phi)=\{I_m\}$.
Therefore, $\phi$ is an isomorphism, and hence, $G$ is liftable.

Because a commuting family of matrices over an algebraically closed field is simultaneously upper-triangularizable, there exists a matrix $B\in \text{GL}(m,k)$ such that 
\[BG_{0}B^{-1}\subset \mr{UT}(m,k).\]
By Lemma \ref{4.2}, all eigenvalues of $A\in G_0$ are 1. Therefore, $BG_0B^{-1}\subset \text{UUT}(m,k)$.
Let $\beta:=[B]\in\text{PGL}(m,k)$.
Then $pr\mid_{BG_0B^{-1}}\co BG_0B^{-1}\to \beta G\beta^{-1}$ is an isomorphism.

We show Proposition \ref{4.5} $(2)$.
We assume that $G\cong (\mathbb{Z}/p\mathbb{Z})^{\oplus u}\rtimes\mathbb Z/l\mathbb Z$ where $\gcd(p,l)=1$.
Let $G_p$ be a Sylow $p$-subgroup of $G$.
By Proposition \ref{4.5} $(1)$, there exists a group $H\subset {\rm GL}(m,k)$ such that 
\[pr\mid_{H}\co H\to G_p\] is an isomorphism.
Let $A\in {\rm GL}(m,k)$ be a matrix such that $[A]\in G$ is an element of order $l$.
As in the proof of Lemma \ref{4.4}, we may assume \( A^l = I_m \) after multiplying by a suitable constant if necessary.
Let 
\[H_A:=\langle H,A\rangle \subset {\rm GL}(m,k)\]
be the group generated by $H$ and $A$.
Then the homomorphism $pr\mid_{H_A}\colon H_A\to G$ is surjective.
By Lemma \ref{4.3}, and $G_p$ is a normal subgroup of $G$, we get that $AHA^{-1}\subset H$.
Thus, $H$ is a normal subgroup of $H_A$.
Since $\gcd(p,l)=1$, we get that $H\cap \langle A\rangle=\{I_m\}$.
Thus, $|H_A|=p^ul$.
This implies that $pr\mid_{H_A}$ is isomorphic, and $G$ is liftable.
\end{proof}
In characteristic zero, the paper~[\ref{bio:z22}] investigates abelian groups 
acting on smooth projective hypersurfaces that are liftable to the general 
linear group. On the other hand, the paper~[\ref{bio:th25n}] provides a 
classification of abelian groups acting on smooth hypersurfaces in 
$\mathbb{P}^3$ that are not liftable.
\begin{thm}\label{4.7}
Let $\mathcal{X}\subset \mathbb{P}^{n+1}$ be an irreducible hypersurface of degree $g$ such that $P=[1:0:\cdots:0]$ is an extendable wild Galois point for $\mathcal{X}$, 
let $m$ be the multiplicity of $\mathcal{X}$ at $P$, and
let $G_{\pi_P}$ be the Galois group of the projection $\pi\co \mac X\da\mathbb P^n$.
We assume that there exists a group  $G\subset \mathrm{UT}(*,n+1)$ such that the homomorphism $pr\mid_G\co G\to G_{\pi_P}$ is an isomorphism, and
$G = \Biggl\langle 
\begin{pmatrix}
	e_{d-m} & 0 \\
	0 & I_{n+1}
\end{pmatrix},\ G_p 
\Biggr\rangle$
where $G_p$ is a Sylow $p$-subgroup of $G$. 
Let $F(X_0,\ldots,X_{n+1})$ be the defining equation of $\mathcal{X}$.
Then there exist $A(X_1,\ldots, X_{n+1}) \in k[X_1,\ldots,X_{n+1}]_{m}$ and 
$B(X_1,\ldots, X_{n+1}) \in k[X_1,\ldots,X_{n+1}]_{d}$ such that
	\[
	F(X_0,\ldots,X_{n+1})
	= A(X_1,\ldots, X_{n+1})\left(\prod_{g\in G} g^*X_0\right) + B(X_1,\ldots, X_{n+1}).
	\]
\end{thm}
\begin{proof}
Since the multiplicity of $\mathcal X$ at \(P\) is \(m\), and $\mac{X}$ is irreducible,
the defining equation of \(\mathcal{X}\) is written as
\[
F(X_0,\ldots, X_{n+1})
=\sum_{i=0}^{d-m}F_{d-i}(X_1,\ldots, X_{n+1})X_0^i
\]
where \(F_{d-i}(X_1,\ldots, X_{n+1})\in k[X_1,\ldots, X_{n+1}]_{d-i}\) for $i=0,\ldots,d-m$, and for $u=d-m$ and $0$ we have \(F_{d-u}(\X)\neq 0\).
Since $G_{\pi_P}\subset \mr{Aut}(\mac X)$, we have that
for $g\in G$ there exists an element $t_g\in k^{\times}$ such that $g^*F(X_0,\ldots, X_{n+1})=t_gF(X_0,\ldots, X_{n+1})$.
Since $d-m=p^{u}l$, we get that if $g \in \left\langle 
\begin{pmatrix}
	e_{d-m} & 0 \\
	0 & I_{n+1}
\end{pmatrix}
\right\rangle$,
then 
\[
g^*(F_m(\X) X_0^{d-m}) = F_m(X_1,\ldots,X_{n+1}) X_0^{d-m}.
\]
This implies $t_g=1$.
By Lemma~\ref{4.0} (2), if $g \in G_p$, then 
\[
g^*X_0 = X_0 + \sum_{i=1}^{n+1} a_i X_i, \qquad a_i \in k.
\]
In particular, since $g \in \mathrm{UT}(*,I_{n+1})$, it follows that 
\[
g^*X_i = X_i \quad \text{for } i=1,\ldots,n+1.
\]
Hence, we conclude that the coefficient of $X_0^{d-m}$ in $g^*F(X_0,\ldots,X_{n+1})$ 
remains equal to $F_m(\X)$.
Therefore, $t_g=1$.
Since $G = \Biggl\langle 
\begin{pmatrix}
	e_{d-m} & 0 \\
	0 & I_{n+1}
\end{pmatrix},\ G_p 
\Biggr\rangle$, we get that $t_g=1$ for any $g\in G$.
Let \(K\) be an algebraic closure of the function field \(k(X_1,\ldots,X_{n+1})\).
Then we may factor
\[
\sum_{i=1}^{d-m}F_{d-i}(X_1,\ldots, X_{n+1})X_0^{i}
=F_m(X_1,\ldots, X_{n+1})\prod_{j=1}^{d-m}(X_0-H_j)
\]
with \(H_j\in K\).
Since the left-hand side is divisible by \(X_0\), there exists some \(j\) with \(H_j=0\).
Without loss of generality, we may assume \(H_1=0\).
Hence we obtain
\[
F(X_0,\ldots,X_{n+1})
=F_m(X_1,\ldots,X_{n+1})\left(X_0\prod_{j=2}^{d-m}(X_0-H_j)\right)
+F_d(X_1,\ldots,X_{n+1}).
\]
Since \(G\) acts trivially on $K$ , for any \(g\in G\)
\begin{dmath*}
g^*\!\Bigl(F_m(X_1,\ldots,X_{n+1})\,X_0\prod_{j=2}^{d-m}(X_0-H_j)\Bigr)=F_m(X_1,\ldots,X_{n+1})\left(g^*X_0\right)\prod_{j=2}^{d-m}(g^*X_0-H_j)=F_m(X_1,\ldots,X_{n+1})\,X_0\prod_{j=2}^{d-m}(X_0-H_j).
\end{dmath*}
Then
\[g^*X_0=X_0-H_j\ \text{for some }j.\]
Since $G\subset \mr{UT}(*,I_{n+1})$, $g^*X_0=X_0$ if and only if $g=I_{n+2}$.
Since \(|G|=d-m\), it follows that
\[
X_0\prod_{j=2}^{d-m}(X_0-H_j)
=\prod_{g\in G}g^*X_0.
\]
Therefore, 
\[
F(X_0,\ldots,X_{n+1})
=F_m(X_1,\ldots,X_{n+1})\left(\prod_{g\in G}g^*X_0\right)+F_d(X_1,\ldots,X_{n+1}).
\]
\end{proof}
We show the only if part of Theorem~\ref{1.2} (2) in Theorem~\ref{3.8}.
\begin{thm}\label{3.8}
Let \( G \subset \mathrm{PGL}(n+2,k) \) be a group isomorphic to \( (\mathbb{Z}/p\mathbb{Z})^{\oplus u} \rtimes \mathbb{Z}/l\mathbb{Z} \) where \( \gcd(p, l) = 1 \) if $l\geq2$.
We assume that here exists a group \( H \subset \mathrm{GL}(n+2,k) \) such that \( pr\mid_{H} \colon H \to G \) is an isomorphism, and \( H \) satisfies the following two conditions:
\begin{itemize}
\item[$(i)$] \( \dim \mathrm{im}(A - I_{n+2}) = 1 \) for some \( A \in H \);
\item[$(ii)$] \( \mathrm{im}(B - I_{n+2}) = \mathrm{im}(C - I_{n+2}) \) for all \( B, C \in H \setminus \{I_{n+2}\} \).
\end{itemize}
Then there exists an irreducible hypersurface \( \mac X \subset \mathbb{P}^{n+1} \) with an extendable wild Galois point $P\in\mb P^{n+1}$ such that \( G \) is the Galois group of the projection $\pi_P \colon \mac X \dashrightarrow \mathbb{P}^n $.
\end{thm}
\begin{proof}
Let $H_p$ be a Sylow $p$-subgroup of $H$.
Since $H_p$ is an abelian group,
by simultaneously upper- triangularized,
there exists a matrix $Q\in \text{GL}(n+2,k)$ such that $QH_pQ^{-1}\subset \mr{UT}(n+2,k)$.
Since $\rmc(k)=p$, $QH_pQ^{-1}\subset \mr{UUT}(n+2,k)$.
Since \( Q \) is an invertible matrix, the conjugate \( Q H Q^{-1} \) satisfies conditions \((i)\) and \((ii)\). Hence, without loss of generality, we may assume from the outset that
$H_{p} \subset \mathrm{UUT}(n+2,k)$.

Let \( e_{1}, \dots, e_{n+2} \) be the standard basis of \( k^{n+2} \).  
Since \( H_{p} \subset \mathrm{UUT}(n+2,k) \), by conditions \((i)\) and \((ii)\), after renumbering the basis vectors if necessary, 
we may assume that
\[
\mathrm{im}(A - I_{n+2}) = \langle e_{1} \rangle_k
\]
for all  $A \in H \backslash\{ I_{n+2} \}$
where \( \langle e_{1} \rangle_k \) denotes the one-dimensional subspace generated by \( e_{1} \).
Hence every \( A\in H \) has the form
\[
A=\begin{pmatrix}
	a_{11} & a_{12} & \cdots & a_{1n+2}\\
	0 & 1 & \cdots & 0\\
	\vdots & \vdots & \ddots &\vdots\\
	0 & 0 &  \cdots & 1
\end{pmatrix}\in\mathrm{UT}(*,I_{n+1}).
\]
Thus, $H\subset \mr{UT}(*,I_{n+1})$.
By Lemma \ref{4.0}, an element \( A\in H \) lies in \( H_p \) if and only if its \((1,1)\)-entry is \( 1 \).
Since \( H\cong (\mathbb{Z}/p\mathbb{Z})^{\oplus u} \rtimes \mathbb{Z}/l\mathbb{Z} \), there exists a matrix $B\in H$ such that $B^l=I_{n+2}$ and $B_l^i\neq I_{n+2}$ for $1\leq i\leq l-1$. 
We denote the $(1,i)$-entry of the matrix $B$ by $b_{1i}$ for $i=1,\ldots, n+2$.
Note that $b_{11}$ is a primitive $l$-th root of unity, and $H=\langle B,H_p \rangle$.
We set $C:=
\begin{pmatrix}
1 & \frac{b_{12}}{b_{11}-1} & \cdots & \frac{b_{1n+2}}{b_{11}-1} \\
	0 & 1 & \cdots & 0\\
	\vdots & \vdots & \ddots &\vdots\\
	0 & 0 &  \cdots & 1
\end{pmatrix}\in \mathrm{UT}(*,I_{n+1})$.
Then $CBC^{-1}=\begin{pmatrix}
	b_{11} & 0\\
	0 & I_{n+1}
\end{pmatrix}$
 and 
 $CHC^{-1}\subset \mr{UT}(*,I_{n+1})$.
 We set $H':=CHC^{-1}$ and $D:=CBC^{-1}$.
Let $N:=p^ul$, and
\[F_{H'}(X_0,\ldots,X_{n+1}):=\prod_{h\in H'}h^*X_0\in k[X_0,\ldots,X_{n+1}]_N.\]
Since for $h\in H$, $h^*X_0=tX_0$ for some $t\in k^{\times}$ if and only if $h\in\langle D\rangle $,
we get that 
there exists a nonzero form $G(X_0,\ldots, X_{n+1})\in k[X_0,\ldots,X_{n+1}]_{N-l}$ of degree $N-l$ such that
\[F_{H'}(X_0,\ldots,X_{n+1})=G(X_0,\ldots, X_{n+1})X_0^l.\]
Moreover, the coefficient of $X_0^l$ in $F_{H'}(X_0,\ldots,X_{n+1})$ is a nonzero form.

We now prove that for $m \geq 0$ there exist forms $A(\X)$ of degree $m$ 
and $B(\X)$ of degree $N+m$ such that  
\[A(\X)F_{H'}(X_0,\ldots,X_{n+1})+B(\X)\] is irreducible. 
The proof is divided into two cases: first, the case $n \geq 2$, and then the case $n=1$.

We assume that $n\geq 2$.
We take irreducible forms $A(X_1,\ldots, X_{n+1})$ of degree $m$ and $B(X_1,\ldots,X_{n+1})$ of degree $N+m$.
We set
\[F(X_0,\ldots, X_{n+1}):=A(X_1,\ldots,X_{n+1})F_{H'}(X_0,\ldots,X_{n+1})+B(X_1,\ldots, X_{n+1}).\]
If $F(X_0,\ldots,X_{n+1})$ is reducible, then we may write
\[
F(X_0,\ldots,X_{n+1}) = Y_1(X_0,\ldots,X_{n+1}) \, Y_2(X_0,\ldots,X_{n+1})
\]
where $Y_1(X_0,\ldots,X_{n+1})$ and $Y_2(X_0,\ldots,X_{n+1})$ are forms of degree at least one. 
The constant term $B(X_1,\ldots,X_{n+1})$ of $F(X_0,\ldots,X_{n+1})$ with respect to $X_0$
is given by the product of the constant terms of 
$Y_1(X_0,\ldots,X_{n+1})$ and $Y_2(X_0,\ldots,X_{n+1})$ with respect to $X_0$.
Since $B(X_1,\ldots,X_{n+1})$ is irreducible, this leads to a contradiction. 
Therefore $F(X_0,\ldots,X_{n+1})$ is irreducible.

We assume that $n=1$.
Suppose $l \geq 2$.  
Let $D(X_1,X_2)$ be the coefficient of $X_0^l$ in $F_{H'}(X_0,X_1,X_2)$. 
Then $D(X_1,X_2)$ is a form of degree $N-l$.  
We take $(a_1,b_1),(a_2,b_2)\in k^2$ such that  
\[
a_1 b_2 \neq a_2 b_1\quad\text{and}\quad D(a_i,b_i)\neq 0
\] for $i=1,2$.  
Let $A(X_1,X_2)$ be a homogeneous polynomial of degree $m$ satisfying $A(a_i,b_i)\neq 0$ for $i=1,2$.  
We set
\[
F(X_0,X_1,X_2)
:= A(X_1,X_2)F_{H'}(X_0,X_1,X_2) 
+ (b_1X_1-a_1X_2)^{N+m-1}(b_2X_1-a_2X_2).
\]
If $F(X_0,X_1,X_2)$ is reducible, then 
\[
F(X_0,X_1,X_2)=Y_1(X_0,X_1,X_2)\,Y_2(X_0,X_1,X_2).
\]
where $Y_1(X_0,X_1,X_2)$ and $Y_2(X_0,X_1,X_2)$ are forms of degree at least one. 
For $i=1,2$, let $B_i(X_1,X_2)$ denote the constant term of 
$Y_i(X_0,X_1,X_2)$ with respect to the variable $X_0$.
Then
\[
(b_1X_1-a_1X_2)^{N+m-1}(b_2X_1-a_2X_2)=B_1(X_1,X_2)B_2(X_1,X_2).
\]
Hence, after renumbering if necessary, we may assume that 
\[B_1(X_1,X_2)=t(b_1X_1-a_1X_2)^M\]
and
\[B_2(X_1,X_2)=t^{-1}(b_1X_1-a_1X_2)^{N+m-M-1}(b_2X_1-a_2X_2)\]
where $t\in k^{\times}$ and $M$ is the degree of $Y_1(X_0,X_1,X_2)$.  
Consider the element
\[
h=\begin{pmatrix} 
	e_l & 0 & 0 \\ 
	0 & 1 & 0 \\ 
	0 & 0 & 1
\end{pmatrix}\in H'.
\]
The action $h^*$ sends $X_0 \mapsto e_l X_0$ and acts trivially on $X_1,X_2$.  
Thus, 
\[h^*F(X_0,X_1,X_2)=F(X_0,X_1,X_2).\]
It follows that $h^*Y_1(X_0,X_1,X_2)=sY_1(X_0,X_1,X_2)$ or $sY_2(X_0,X_1,X_2)$ for some $s\in k^{\times}$.
Since $a_1 b_2 \neq a_2 b_1$,
\[
B_1(X_1,X_2)\neq s'B_2(X_1,X_2)\quad\mr{for\ any}\ s'\in k^{\times}.
\]
Thus,
\[
h^*Y_1(X_0,X_1,X_2) = s_1Y_1(X_0,X_1,X_2)
\]
and
\[
h^*Y_2(X_0,X_1,X_2) = s_2Y_2(X_0,X_1,X_2)
\]
for some $s_1,s_2\in k^{\times}$.  
Moreover, since $h^*$ sends $X_0 \mapsto e_l X_0$ and acts trivially on $X_1,X_2$, 
 it follows that $s_i=1$, i.e. 
 \[h^*Y_i(X_0,X_1,X_2) = Y_i(X_0,X_1,X_2)\]
 for $i=1,2$.  
Hence the coefficients of $X_0^j$ in $Y_i(X_0,X_1,X_2)$ vanish whenever $j$ is not a multiple of $l$.  
From the equality $F(X_0,X_1,X_2)=Y_1(X_0,X_1,X_2)Y_2(X_0,X_1,X_2)$,
we deduce that
\[
A(X_1,X_2)D(X_1,X_2) 
= C_2(X_1,X_2)B_1(X_1,X_2)
+ C_1(X_1,X_2)B_2(X_1,X_2)
\]
where $C_i(X_1,X_2)$ is the coefficient of $X_0^l$ in $Y_i(X_0,X_1,X_2)$ for $i=1,2$.  
Substituting $(X_1,X_2)=(a_1,b_1)$, the left-hand side is nonzero while the right-hand side vanishes, which is a contradiction.  
Therefore $F(X_0,X_1,X_2)$ is irreducible.  
For the case $l=1$, define
\[
F(X_0,X_1,X_2)
:= A(X_1,X_2) F_{H'}(X_0,X_1,X_2) 
+ (b_1X_1-a_1X_2)^{N+m},
\]
where $A(a_1,b_1)\neq 0$ and $D(a_1,b_1)\neq 0$.  
Suppose for contradiction that $F(X_0,X_1,X_2)$ is reducible.  There are forms $Y_1(X_0,X_1,X_2)$ and $Y_2(X_0,X_1,X_2)$ of positive degree such that
$F(X_0,X_1,X_2)$
$=Y_1(X_0,X_1,X_2)Y_2(X_0,X_1,X_2)$.
The constant terms in $X_0$ of $Y_1(X_0,X_1,X_2)$ and $Y_2(X_0,X_1,X_2)$ are powers of $(b_1X_1-a_1X_2)$.  
Comparing the coefficients of $X_0$ after substituting $(X_1,X_2)=(a_1,b_1)$ yields a contradiction.  
Thus $F(X_0,X_1,X_2)$ is also irreducible in this case.

It follows from the above discussion that for every $m\geq0$ there exists an irreducible form $F(X_0,\ldots, X_{n+1})\in k[X_0,\ldots,X_{n+1}]_{N+m}$ such that
\[h^*F(X_0,\ldots,X_{n+1})=F(X_0,\ldots,X_{n+1})\]
 for $h\in H'$ , and $F(X_0,\ldots,X_{n+1})$ is a polynomial in $X_0$ of degree $N$.
Let $\mac Y\subset \mathbb P^{n+1}$ be an irreducible hypersurface defined by $F(X_0,\ldots,X_{n+1})=0$.
We set $P':=[1:0:\cdots:0]\in\mathbb P^{n+1}$.
Then the multiplicity of $\mac Y$ at $P'$ is $m$, and $P'$ is an extendable wild Galois point for $\mac Y$. Moreover, the Galois group of the projection $\pi_{P'}\colon \mac Y\da \mathbb P^n$ is $[Q]G[Q]^{-1}$.
Consequently, \( G \) is the Galois group of the wild Galois point $P:=[R](P')$ for the irreducible hypersurface $\mac X:=[Q](\mac Y)$ with the multiplicity $m$ at $P$.
\end{proof}

\section{Wildly ramified Galois point}
Let $X$ be a normal algebraic variety.
Let $D$ be a prime divisor, i.e. a codimension one irreducible closed subset.
Then 
\[
\mathcal{O}_{X,D} :=\varinjlim_{D \cap U \ne \emptyset} \mathcal{O}_X(U)
\]
is a discrete valuation ring
where \(\mathcal{O}_X(U)\) denotes the ring of regular functions on an open subset \(U \subset X\).  
Let \(\mathfrak{m}_D\) denote its maximal ideal, and let \(v_D\) be the corresponding discrete valuation.
Let \( f\co X \dashrightarrow Y \) be a dominant Galois rational map between normal algebraic varieties of the same dimension over \( k \). Let \( K := k(X) \), \( L := k(Y) \), and let \( G := \mathrm{Gal}(K/L) \) be the Galois group of the field extension $K/L$.
For each integer \( i \geq 0 \), we define the \emph{\textbf{\( i \)-th ramification group at the prime divisor \( D \subset X \)}} by
\[
G_i^D := \left\{ g \in G^D \mid g(x) \equiv x \bmod \mathfrak{m}_D^{i+1} \ \text{for all } x \in \mathcal{O}_{X,D} \right\},
\]
where
\[
G^D := \left\{ g \in G \mid g(\mathcal{O}_{X,D}) = \mathcal{O}_{X,D} \right\}.
\]
In this setting:
\begin{itemize}
\item \( G_0^D \) is called the \emph{\textbf{inertia group}} at \( D \),
\item \( G_1^D \subset G_0^D \) is called the \emph{\textbf{wild inertia group}} at \( D \).
\end{itemize}
We say that \( f \) is \emph{\tb{wildly ramified at \( D \)}} if \( G_1^D \ne 1 \), i.e., the wild inertia group is nontrivial. 
In particular, \( f \) is said to be \emph{\tb{wildly ramified}} if it is wildly ramified at some prime divisor on \( X \).
\begin{thm}\label{3.1}$([\ref{bio:jn}, \S10.\ \mathrm{Proposition}\ 10.2])$.
Let $k$ be an algebraically closed field of positive characteristic $p>0$.
	Let \( f\co X \dashrightarrow Y \) be a dominant Galois rational map between normal algebraic varieties of the same dimension, and let \( G\) be the Galois group of the field extension $k(X)/k(Y)$.
	Let $D$ be a prime divisor on $X$ such that $G^D\subset G$ is a non-trivial group.
	Then:
\begin{enumerate}
\item[$(1)$] The quotient group \( G_0^D / G_1^D \) is cyclic of order prime to \( p \).
\item[$(2)$] For each \( i \geq 1 \), the quotient group \( G_i^D / G_{i+1}^D \) is a finite direct product of cyclic groups of order \( p \).
\end{enumerate}
\end{thm}
Let $X$ be an irreducible algebraic variety.
For \( g \in \mathrm{Bir}(X) \), let \( \mathrm{Dom}(g) \subset X \) denote the maximal open subset on which \( g \) is defined as a morphism, and define
\[
\mathrm{Fix}(g) := \overline{ \{ x \in \mathrm{Dom}(g) \mid g(x) = x \} },
\]
where the closure is taken in the Zariski topology on \( X \).
Let $G\subset {\rm Bir}(X)$ be a finite group, and let \( D \subset X \) be a prime divisor.
Define
\[
G_D := \{ g \in G \mid D \subset \mathrm{Fix}(g) \}.
\]
We assume that $X$ is normal.
Then the local ring $\mathcal{O}_{X,D}$ is a discrete valuation ring.
For $g\in \mathrm{Bir}(X)$, we have $g^*(\mathcal{O}_{X,D})$ if and only if $\overline{g(D\cap \mr{Dom}(g))}=D$.
With this observation, we obtain the following criterion:
\begin{pro}\label{3.2}
Let $k$ be an algebraically closed field of positive characteristic $p>0$.
Let \( f\co X \dashrightarrow Y \) be a dominant Galois rational map between normal projective varieties of the same dimension over \( k \), and let \( G\subset \mathrm{Bir}(X) \) be the Galois group of the field extension $k(X)/k(Y)$.
Let \( D \subset X \) be a prime divisor.
Then the following are equivalent:
\begin{enumerate}
	\item[$(1)$] The order of \( G_D \) is divisible by \( p \).
	\item[$(2)$] The first ramification group \( G_1^D \) at \( D \) is nontrivial.
\end{enumerate}
\end{pro}
\begin{proof}
We assume that the order of \( G_D \) is divisible by \( p \).
Then there exists a birational self-map $g\in G_D$ of order $p$.
Since $D\subset \mathrm{Fix}(g)$, $\ol{g(D)}=D$.
It follows that $g^*\in G^D$ and ${\rm Dom}(g)\cap D\ne\emptyset$.
Let $\alpha$ be an element of the local ring $\mathcal{O}_{X,D}$.  
We take a pair $(f,W)$ representing $\alpha$ where $W$ is an open subset of $X$ with $W\cap D\neq \emptyset$, and $f \in \mathcal{O}_X(W)$.  
	Since $X$ is a normal projective variety, there exist open subsets $U,V\subset X$ such that 
	$g\mid_U\co U\to V$ is an isomorphism, and the codimension of $X\backslash V$ is at least $2$.
Hence $V\cap D\ne\emptyset$, and therefore
	\[
	(g|_{\mathrm{Dom}(g)})^{-1}(W)\cap D \neq \emptyset.
	\]  
We set 
\[
O := (g|_{\mathrm{Dom}(g)})^{-1}(W), \quad {\rm and}\quad h := (f - f\circ g)|_O \in \mathcal{O}_X(O).
\]  
Since $\overline{g(D)} = D$, it follows that $h(x)=0$ in the stalk $\mathcal{O}_{X,x}$ for every 
	$x\in D\cap O$. Hence, the pair $(h,O)$ represents an element of the maximal ideal 
	$\mathfrak{m}_D \subset \mathcal{O}_{X,D}$.  
	Consequently, we deduce that
	\[
	g^*(\alpha) \equiv \alpha\bmod{\mathfrak{m}_D}.
	\]
Therefore, $g^* \in G^D_0$.
Since the order of $g$ is $p$, it follows that the order $G_0^D$ is divisible by $p$.
By Theorem \ref{3.1} $(1)$, $G_1^D$ is divisible by $p$.
Thus, $G_1^D$ is nontrivial.
	
We assume that the group $G_1^D$ is nontrivial.
	By Theorem \ref{3.1} $(2)$, there exists a birational self-map $g\in\rm{Bir}(X)$ of order $p$ such that
	$g^*\in G_1^D$.
	If $g\not\in G_D$, then there is a point $z\in {\rm Dom}(g)\cap D$ such that $g(z)\neq z$.
Since $X$ is projective, we can find an open subset $W\subset X$ and a regular function $f\in \mac O_X(W)$ such that $W\cap D\ne\emptyset$, 
$f(z)=0$ in $\kappa(z)$, and 
$f(g(z))\neq 0$ in $\kappa(g(z))$
where $\kappa(z)$ and $\kappa(g(z))$ are the residue fields of the local rings $\mathcal{O}_{X,z}$ and $\mathcal{O}_{X,g(z)}$, respectively.
Let \(\alpha\) be the equivalence class of the pair \((f, W)\) in the local ring $\mathcal{O}_{X,D}$.
Since $z,g(z)\in D$, 
	\[
	g^*(\alpha)\not\equiv \alpha \bmod \mathfrak{m}_D.
	\]
Hence, $g^*\not\in G_0^D$.
On the other hand, since $g^* \in G_1^D \subset G_0^D$, 
this yields a contradiction.
Thus $g \in G_D$, and therefore the order of $G_D$ is divisible by $p$.
\end{proof}
\begin{thm}\label{3.3}
Let $k$ be an algebraically closed field of positive characteristic $p>0$.
Let \( \mathcal{X} \subset \mathbb{P}^{n+1} \) be a normal hypersurface of degree \( d \) over $k$. Suppose that \( P \in \mathbb{P}^{n+1} \) is an extendable Galois point for \( \mathcal{X} \), with multiplicity \( m \) at \( P \).
If \( d - m \) is a multiple of \( p \), then the projection \( \pi_P \colon \mathcal{X} \dashrightarrow \mathbb{P}^n \) is wildly ramified.
\end{thm}
\begin{proof}
Let $G_{\pi_P}\subset \rm{Bir}(\mathcal X)$ be the group defined in $(\ref{1})$.
Then $G_{\pi_P}$ is the Galois group of $\pi_P$ such that $|G_{\pi_P}|=d-m$.
Since $P$ is extendable,  $G_{\pi_P}\subset \mathrm{PGL}(n+2,k)$.
Since $p$ divides $d-m$,
	there exists an automorphism $g\in G_{\pi_P}$ of order $p$.
	By replacing the coordinate system if necessary, we may assume that $P = [1 : 0 : \cdots : 0]$.
	By Lemma~\ref{2.1}, there exists a matrix $A=\begin{pmatrix} 
		1 & a_{12} & \dots  & a_{1n+2} \\
		0 & 1 & \dots  & 0\\
		\vdots & \vdots & \ddots & \vdots \\
		0 & 0& \dots  & 1
	\end{pmatrix}\in\mr{UT}(*,I_{n+1})$ such that  $g=[A]$.
	Let $\mac H\subset \mb P^{n+1}$ be the hypersurface defined by 
	\[\sum_{i=2}^{n+2}a_{1i}X_{i-1}=0.\]
	This hypersurface $\mathcal H$ is the projective subspace determined by the eigenspace of the matrix \( A \) corresponding to the eigenvalue \( 1 \).
	Then $\mac X\cap \mac H\subset\mr{Fix}(g)$.
Since $\mathcal X$ and $\mathcal H$ are integral (i.e.\ irreducible and reduced)  hypersurfaces in $\mathbb P^{n+1}$, 
$\mathcal X\cap \mathcal H$ has codimension $2$ in $\mathbb P^{n+1}$.
Thus, there exists an irreducible component $D$ of $\mac X\cap \mac H$ with codimension one in $\mathcal X$.
Then $D$ is a prime divisor on $\mac X$ such that $g\in G_D$.
	By Proposition \ref{3.2}, $\pi_P$ is wildly ramified at $D$, and hence $\pi_P$ is wildly ramified.
\end{proof}
One of the central problems in positive characteristic geometry is to determine which finite groups can appear as Galois groups of branched covers. 
Abhyankar formulated a celebrated conjecture concerning the structure of Galois covers in positive characteristic, which was later proved independently by Raynaud $([\ref{bio:ra94},\ \mathrm{Corollary}\ 6.5.3])$ and Harbater $([\ref{bio:ha94},\ \mathrm{Theorem}\  6.2])$. 
\begin{thm}$(\mathrm{Abhyankar's\ Conjecture}\ [\ref{bio:ab57},\ \mathrm{Conjecture}\ 1])$.
Let $G$ be a finite group, and let $p(G)\subset G$ be the subgroup generated by all Sylow $p$-subgroups of $G$.  
Let $X$ be a smooth proper connected curve of genus $g_X$ defined over an algebraically closed field of characteristic $p$, and let $S\subset X$ be a finite subset of cardinality $r_X \ge 1$. Then there exists a $G$-Galois cover 
\[
\varphi \colon Y \to X
\]
branched only at $S$ if and only if the quotient group $G/p(G)$ can be generated by at most $2g_X + r_X - 1$ elements.
\end{thm}
In particular, when $X=\mathbb{P}^1$ and $S$ consists of at least two points, this result characterizes all finite groups $G$ that can occur as Galois groups of covers of $\mathbb{P}^1$ branched only at $S$. This theorem plays a fundamental role in understanding which finite groups can be realized as Galois groups of branched covers of curves in positive characteristic. Motivated by this, we raise the following question in the context of hypersurfaces and Galois points.
\begin{que}
Let $k$ be an algebraically closed field of positive characteristic $p>0$.
	Let $G \cong \mathbb{Z}/p\mathbb{Z}^{\oplus u} \rtimes \mathbb{Z}/l\mathbb{Z}$ where $\gcd(p,l)=1$ and $l$ divides $p^u-1$.  
	Does there exist a smooth hypersurface $\mathcal{X}\subset \mb P^{n+1}$ over $k$ with a Galois point $P\in\mathbb P^{n+1}$ such that the Galois group of the projection
	\[
	\pi_P \colon \mathcal{X} \dashrightarrow \mathbb{P}^n
	\]
	is isomorphic to $G$?
\end{que}
In this paper we do not assume that $\mathcal{X}$ is smooth, so the multiplicity $m$ of $\mathcal{X}$ at $P$ may take various values. 
On the other hand, if $\mathcal{X}$ is smooth, then the multiplicity $m$ of $\mathcal{X}$ at any point is necessarily either $0$ or $1$. 
Therefore, in the smooth case, it suffices to restrict our consideration to the cases $m=0$ and $m=1$.

\end{document}